\documentclass[letter]{amsart}
\usepackage{tikz}
\usetikzlibrary{matrix,arrows}
\usepackage{amssymb, mathtools, enumitem, hyperref, cleveref}

\DeclareMathOperator{\rad}{rad}
\DeclareMathOperator{\nil}{nil}

\DeclareMathOperator{\Spec}{Spec}
\DeclareMathOperator{\mSpec}{mSpec}
\DeclareMathOperator{\Min}{Min}
\newcommand{\Z}{\mathbb{Z}}
\newcommand{\N}{\mathbb{N}}
\newcommand{\m}{\mathfrak{m}}
\newcommand{\p}{\mathfrak{p}}

\newtheorem{theorem}{Theorem}
\numberwithin{theorem}{section}
\newtheorem{prop}[theorem]{Proposition}
\newtheorem{lemma}[theorem]{Lemma}
\newtheorem{cor}[theorem]{Corollary}
\theoremstyle{definition}
\newtheorem*{mydef}{Definition}
\newtheorem{ex}{Example}
\newtheorem{rem}[theorem]{Remark}
\crefname{question}{question}{questions}
\everymath{\displaystyle}

\begin{document}
\title{Surjections of unit groups and semi-inverses}
\author{Justin Chen}


\begin{abstract}
Given a surjective ring homomorphism, we study when the induced group homomorphism 
on unit groups is surjective. To this end, we introduce notions of generalized inverses and 
units, as well as a class of rings such that the set of closed points in the spectrum is a 
closed set. It is shown that any surjection out of such a ring induces a surjection on 
unit groups.
\end{abstract}

\maketitle

\section{Introduction}
Let \textbf{CRing} be the category of commutative rings with $1 \ne 0$, and
\textbf{Ab} the category of abelian groups. One of the most natural
functors from \textbf{CRing} to \textbf{Ab} is the group of units functor, $(\_)^\times$,
associating to any (commutative) ring its (abelian) group of units. Functoriality 
follows from the fact that a ring homomorphism $\varphi : R \to S$ sends $1$ to $1$,
hence units to units, and thus induces (by set-theoretic restriction) a group homomorphism 
$\varphi^\times : R^\times \to S^\times$. By definition as a set-theoretic restriction, one sees 
that $\varphi$ injective implies $\varphi^\times$ injective (i.e., $(\_)^\times$ is ``left exact''). 
The question we now consider is: when does $\varphi$ surjective imply 
$\varphi^\times$ surjective, i.e., how does $(\_)^\times$ fail to be ``right exact''?

\begin{ex}
For any prime number $p$, the natural surjection $\Z \twoheadrightarrow \Z/p\Z$ induces 
a group homomorphism $\Z/2\Z \cong \Z^\times \to (\Z/p\Z)^\times \cong \Z/(p-1)\Z$, which 
is a surjection iff $p = 2, 3$. 
\end{ex}

\begin{ex}
For a field $k$, any ring surjection $\varphi : k \twoheadrightarrow R$ is necessarily injective,
hence an isomorphism, so (by functoriality) $\varphi^\times$ is also an isomorphism.
\end{ex}

\begin{ex}
For a field $k$, the surjection $\varphi_1 : k[x] \twoheadrightarrow k[x]/(x) \cong k$ 
induces a surjection on unit groups, but $\varphi_2 : k[x] \twoheadrightarrow k[x]/(x^2)$ does 
not, as $\varphi_2(1 + x) \in (k[x]/(x^2))^\times$, but is not the image of any unit of $k[x]$ 
($ =$ nonzero constant in $k$).
\end{ex}

With these examples at hand, we make the following (non-vacuous) definition:

\begin{mydef}
A ring surjection $\varphi : R \twoheadrightarrow S$ has $(*)$ if $\varphi^\times : R^\times 
\twoheadrightarrow S^\times$ is surjective. We say that the ring $R$ has $(*)$ if every ring 
surjection $\varphi : R \twoheadrightarrow S$ (for any ring $S$) has $(*)$.
\end{mydef}

If $\varphi : R \twoheadrightarrow S$ is a ring surjection, then $S \cong R/I$ for some $R$-ideal
$I$ (namely $I = \ker \varphi$), so one may instead refer to an ideal $I$ having $(*)$ (i.e. if the 
canonical surjection $R \twoheadrightarrow R/I$ has $(*)$). Thus $R$ has $(*)$ iff $I$ has $(*)$
for every $R$-ideal $I$, so in this way property $(*)$ for a ring becomes an
ideal-theoretic statement. The examples above say that
any field $k$ has $(*)$, while $\Z$ and $k[x]$ do not.

We begin with some characterizations of $(*)$. Recall that if $W$ is a multiplicative set, the
\textit{saturation} of $W$ is defined as $W^\sim := \{r \in R \mid \exists s \in R, sr \in W\}$, 
and $W$ is called \textit{saturated} if $W = W^\sim$.

\begin{prop} \label{firstChar}
Let $R$ be a ring, $I$ an $R$-ideal. The following are equivalent:

\indent i) $I$ has $(*)$

\indent ii) $R^\times + I$ is saturated

\indent iii) $R^\times + I = (1 + I)^\sim$

\indent iv) For any $a \in R$ such that $1 - ab \in I$ for some $b \in R$, there exists
$u \in R^\times$ with $1 - au \in I$.
\end{prop}

\begin{proof}
(ii) $\iff$ (iii): follows from the containment $1 + I \subseteq R^\times + I \subseteq 
(1 + I)^\sim$ which holds for any ideal $I$, and the fact that saturation is a 
closure operation (in particular, is monotonic and idempotent).

(i) $\implies$ (iii): Suppose that the canonical surjection $p : R \twoheadrightarrow R/I$
induces a surjection $p^\times : R^\times \twoheadrightarrow (R/I)^\times$, i.e. if 
$r \in R$ is such that $p(r)$ is a unit, then $p(r) = p(u)$ for some $u \in R^\times$.
Then $r - u \in \ker p = I$, i.e. $r \in R^\times + I$. Thus the preimage
of the units of $R/I$ is contained in $R^\times + I$, but this preimage is exactly 
$(1 + I)^\sim$, since $p(r)$ is a unit $\iff 1 = p(1) = p(r)p(s)$ for some $s \in R \iff 
1 - rs \in I \iff rs \in 1 + I$.

(iii) $\implies$ (i): if $R^\times + I = (1 + I)^\sim$, then any preimage of a unit of $R/I$
differs from a unit of $R$ by an element of $I$, so every unit of $R/I$ is the image of a 
unit of $R$.

(iii) $\iff$ (iv): Notice that $a \in (1 + I)^\sim \iff 1 - ab \in I$ for some $b \in R$,
and $a \in R^\times + I \iff v - a \in I$ for some $v \in R^\times \iff 1 - v^{-1}a \in I$.
\end{proof}

\section{Sufficient conditions for $(*)$}

As a first application of \Cref{firstChar}, one has the following sufficient condition
for an ideal to have $(*)$ (hereafter, the Jacobson radical of $R$ is denoted by 
$\rad(R) := \bigcap_{\m \in \mSpec(R)} \m$, the intersection of all maximal ideals 
of $R$). 

\begin{cor} \label{radCase}
Let $R$ be a ring, $I$ an $R$-ideal. If $I \subseteq \rad(R)$, then $I$ has $(*)$.
\end{cor}

\begin{proof}
If $I \subseteq \rad(R)$, then $R^\times + I = R^\times = \{1\}^\sim$ is saturated, 
so \Cref{firstChar}(ii) applies.
\end{proof}

In fact, rather than requiring $I$ to be contained in every maximal ideal, 
one can allow finitely many exceptions:

\begin{theorem}
Let $R$ be a ring, $I$ an $R$-ideal. If $I$ is contained in all but finitely many 
maximal ideals of $R$ (i.e. $|\mSpec(R) \setminus V(I)| < \infty$), then $I$ 
has $(*)$.
\end{theorem}

\begin{proof}
Write $\mSpec(R) \setminus V(I) := \{m_1, ..., m_n\}$,
so that $\{I, m_1, ..., m_n\}$ are pairwise comaximal (the case $n = 0$ is 
\Cref{radCase}). Let $p : R \twoheadrightarrow R/I$ be the canonical surjection,
pick $v \in (R/I)^\times$, and write $v = p(r)$ for some $r \in R$. By Chinese 
Remainder, there exists $a \in R$ with $a \equiv 0 \pmod I$, 
$a \equiv 1 - r \pmod {m_i}$ for $i = 1, ..., n$. Since $r$ is not contained in
any maximal ideal containing $I$, $r + a \in R^\times$, and 
$p(r + a) = p(r) = v$.
\end{proof}

\begin{cor} \label{semiCase}
Let $R$ be a semilocal ring, i.e. $|\mSpec(R)| < \infty$. Then $R$ has $(*)$.
\end{cor}

\begin{proof}
If $R$ is semilocal, then for any $R$-ideal $I$, $\mSpec(R) \setminus V(I)$ is finite.
\end{proof}

\Cref{radCase} lends support to the idea that the Jacobson radical will not play a role
in whether or not a ring has $(*)$. This is indeed true, as the following 
reduction to the $J$-semisimple case (i.e. $\rad(R) = 0$) will show.

\begin{prop} \label{reduceToRad}
Let $R$ be a ring, $I$ an $R$-ideal, $p : R \twoheadrightarrow R/I$ the canonical surjection, and 
$\overline{p} : R/\rad(R) \twoheadrightarrow R/(\rad(R) + I)$ the map obtained 
by applying $\_ \otimes_R R/\rad(R)$. Then $p$ has $(*)$ iff $\overline{p}$ 
has $(*)$. In particular, $R$ has $(*)$ iff $R/\rad(R)$ has $(*)$.
\end{prop}

\begin{proof}
Consider the commutative diagram of natural maps
\begin{center}
\begin{tikzpicture}[descr/.style={fill=white,inner sep=2.5pt}]
\matrix (m) [matrix of math nodes, row sep=3em,
column sep=3em]
{ R & R/I \\
R/\rad(R) & R/(\rad(R) + I) \\ };
\path[>=stealth,->>]
(m-1-1) edge node[auto] {$p$} (m-1-2)
(m-2-1) edge node[auto] {$\overline{p}$} (m-2-2)
(m-1-1) edge node[left] {$\alpha$} (m-2-1)
(m-1-2) edge node[auto] {$\beta$} (m-2-2);
\end{tikzpicture}
\end{center}
If $p^\times$ is surjective, then since $\beta^\times$ is surjective (by 
Corollary 2, as $(\rad(R) + I)/I \subseteq \rad(R/I)$), so is $\overline{p}^\times$.
Conversely, suppose $\overline{p}^\times$ is surjective, and let $v \in (R/I)^\times$. 
Then $\beta(v) \in (R/(\rad(R) + I))^\times$, so there exists $\overline{u} \in 
(R/\rad(R))^\times$ with $\overline{p}(\overline{u}) = \beta(v)$. By Corollary 2, 
$\alpha^\times$ is surjective, hence $\overline{u} = \alpha(u)$ for some 
$u \in R^\times$. Then $\beta(p(u)) = \overline{p}(\alpha(u)) = \beta(v)$, so 
$v - p(u) \in \ker \beta$. But $\ker \beta = p(\rad(R))$, so $v - p(u) = 
p(r)$ for some $r \in \rad(R)$. Then $v = p(u + r)$, and 
$u + r \in R^\times + \rad(R) = R^\times$.
\end{proof}

We can use \Cref{reduceToRad} to give examples of rings with $(*)$ that are not 
semilocal. Although the following lemma should be well-known, we include a 
proof for completeness.

\begin{lemma} \label{radOfProd}
For an arbitrary direct product of rings, 
$\rad(\prod_i R_i) = \prod_i \rad(R_i)$.
\end{lemma}

\begin{proof}
$\supseteq$: let $(a_i) \in \prod_i \rad(R_i)$. Then for each $i$ and any
$b_i \in R_i$, $1 - a_ib_i \in R_i^\times$, so every $b = (b_i) \in \prod_i R_i$ 
satisfies $1 - ab = (1 - a_ib_i) \in \prod_i R_i^\times = (\prod_i R_i)^\times$. 

$\subseteq:$ for any surjective ring map $\varphi : R \twoheadrightarrow S$,
$\varphi(\rad(R)) \subseteq \rad(S)$, so applying this to each natural
projection $\pi_j : \prod_i R_i \twoheadrightarrow R_j$ gives 
$\pi_j(\rad(\prod_i R_i)) \subseteq \rad(R_j)$.
\end{proof}

\begin{ex} \label{prodCase}
i) If $R = \prod_i R_i$ is an arbitrary product of semilocal rings, then $R$ has $(*)$ 
(note that such a ring can have infinite Krull dimension, cf. \cite{GH}).
To see this, note that by \Cref{reduceToRad} and \Cref{radOfProd}, it suffices to show 
that any product of fields has $(*)$. Thus, let $R = \prod_i k_i$, where $k_i$
are fields. Using \Cref{firstChar}(iii), let $I$ be an 
$R$-ideal, and $a = (a_i) \in (1+I)^\sim$, such that $1 - ab \in I$ for some $b \in R$. 
Let $J$ be the set of indices $j$ such that $a_j = 0$, and let $e_J$ be
the indicator vector of $J$, i.e. $e_J := (e_i) \in R$, where 
$e_i := \begin{cases}
1, \; i \in J \\
0, \; i \not \in J
\end{cases}$.
Then $e_J(1 - ab) \in I$, and satisfies $(e_J(1 - ab))_i = 0$ iff $i \not \in J$ 
(note that if $i \in J$, then $a_i = 0$ and both $(1 - ab)_i
= 1 - a_ib_i$ and $(e_J)_i$ equal $1$, so their product is also $1 \ne 0$, whereas if 
$i \not \in J$, then $(e_J)_i$ is already $0$). Thus $(a + e_J(1 - ab))_i$ is nonzero 
for every $i$, hence $a + e_J(1 - ab) \in R^\times \implies a \in R^\times + I$.

ii) Via a different approach, we can also show that $(*)$ passes to finite products. 
Let $R = \prod_{i=1}^n R_i$, where $R_i$ have $(*)$. Using \Cref{firstChar}(iv), let 
$I$ be an $R$-ideal. Then $I = \prod_{i=1}^n I_i$ for $R_i$-ideals
$I_i$. Let $a = (a_i) \in R$ be such that $1 - ab \in I$ for some $b = (b_i) \in R$. Then
$1 - a_ib_i \in I_i$ for each $i$, so there exists $u_i \in R_i^\times$ with $1 - a_iu_i \in
I_i$. Thus $u = (u_i) \in \prod_{i=1}^n R_i^\times = R^\times$, and $1 - au \in I$.

iii) In view of \Cref{prodCase}(i), as the diagonal map 
$\Z \hookrightarrow \prod_{p \text{ prime}} \Z/p\Z$ is injective, we see that $(*)$ 
does not pass to subrings. On the other hand, it is easy to see that $(*)$ passes 
to quotient rings.
\end{ex}

We briefly turn to the graded case. Let $R = \bigoplus_{i \ge 0} R_i$ be 
a $\Z_{\geq 0}$-graded ring, $I = \bigoplus_{i \ge 0} I_i$ a graded $R$-ideal, and 
$p : R \twoheadrightarrow R/I$ the canonical surjection, a graded
ring map of degree 0. Let $p_0 : R_0 \twoheadrightarrow (R/I)_0 = R_0/I_0$ be
the induced ring map of degree 0 components. In general, the units of $R$ need 
not be graded. However, with some primality assumptions we may reduce to the 
ungraded case, as follows:

\begin{prop}
Suppose $I$ is prime. If $p_0$ has $(*)$, then $p$ has $(*)$.
The converse holds if $R$ is a domain. 
\end{prop}

\begin{proof}
If $I$ is prime, then $R/I$ is a positively graded domain, which has units
only in degree 0, i.e. $(R/I)^\times \subseteq (R/I)_0$. Then $(R/I)^\times = 
((R/I)_0)^\times = p_0^\times(R_0^\times) \subseteq p^\times(R^\times)$, 
and the first statement follows. 
Conversely, if $R$ is a domain, then $R^\times \subseteq R_0$, so 
$R^\times = (R_0)^\times$ and $p(R^\times) = p_0(R_0^\times)$.
\end{proof}

\begin{cor} \label{gradedCase}
If $I \subseteq R_+ = \bigoplus_{i \ge 1} R_i$ is prime,
then $I$ has $(*)$.
\end{cor}

\begin{proof}
In this case, $I_0 = 0$, so $p_0 : R_0 \rightarrow R_0$ is the identity, hence
$p_0$ has $(*)$. 
\end{proof}

To motivate the next section, we briefly summarize the results thus far: we 
have seen that property $(*)$ for a ring $R$ depends only on the $J$-semisimple
reduction $R/\rad(R)$. Since the $J$-semisimple reduction of a semilocal ring 
is a finite product of fields, this gives an alternate proof of \Cref{semiCase}. However,
being semilocal is not a necessary condition for a ring to have $(*)$, as an 
infinite product of fields is never semilocal. Despite this, the examples given 
so far of rings with $(*)$ are quite similar - e.g. they all share the property 
that the $J$-semisimple reduction is 0-dimensional. 

From a different angle, one can start with the observation that for any ring $R$,
if $r \in R$ is a nonunit, then $R \twoheadrightarrow R/(r^2)$ is such
that $1 + r$ goes to a unit in $R/(r^2)$, with inverse $1 - r$. In particular, if a 
ring $R$ is to have $(*)$, then necessarily any element $r$ must satisfy
$1 + r \in R^\times + (r^2)$, i.e. for any $r \in R$, 
there exists $s \in R$ such that $1 + r - sr^2 \in R^\times$. Recalling that 
$\rad(R) = \{r \in R \mid 1 + (r) \subseteq R^\times \}$,
this will certainly be satisfied if for every $r \in R$,
there exists $s \in R$ with $r - sr^2 = r(1 - sr) \in \rad(R)$. It is this last
condition which we now examine in detail.

\section{Semi-inverses}

Returning to a general setting (laying aside for now the surjectivity question), 
let $R$ be a ring, and $r \in R$. The failure of
$r$ to be a unit is encoded in the set of maximal ideals which contain $r$
-- namely, $r$ is a unit iff $r$ is not contained in any maximal ideal. 
Furthermore, when this occurs there is a unique element $r^{-1}$, with 
$1 - r^{-1} \cdot r = 0 \in \m$ for every maximal ideal $\m$. Generalizing this 
basic fact gives an analogous notion for any $r \in R$:

\begin{mydef}
Let $R$ be a ring, $r \in R$.
A subset $S \subseteq R$ is called a \textit{semi-inverse set for r}
if for every maximal ideal $\m \in \mSpec(R)$, either $r \in \m$, or there 
exists $s \in S$ with $1 - sr \in \m$.
\end{mydef}

Notice that the two cases in the definition above are exhaustive and 
mutually exclusive: i.e. for any $r \in R$ and any $\m \in \mSpec(R)$, it is 
always the case that either $r \in \m$ or there exists $s \in R$ with $1 - sr 
\in \m$, and both cases cannot occur simultaneously. Notice that existence
of semi-inverse sets follows from the Axiom of Choice:
for every maximal ideal $\m$ not containing $r$, the image $\overline{r} \in 
R/\m$ is a unit, so there exists $\overline{s} \in R/\m$ with $\overline{r} \cdot
\overline{s} = \overline{1}$, i.e. $1 - sr \in \m$. This also shows that for any
$r \in R$, the minimum size of a semi-inverse set for $r$ is at most
$|\mSpec(R) \setminus V(r)|$, which leads to the following definition:

\begin{mydef}
For a ring $R$, define a function $\rho : R \to \mathbb{N} \cup \{ \infty \}$ by
$$\rho(r) := 
\begin{cases}
\min\{ |S| : S \text{ semi-inverse set for $r$} \}, & \text{if $r$ has a finite semi-inverse set} \\
\infty, & \text{if $r$ has no finite semi-inverse set}
\end{cases}$$
\end{mydef}

The possible values that the function $\rho$ can attain are rather limited:

\begin{prop} \label{rhoValues}
Let $R$ be a ring, $r \in R$. Then $\rho(r) < \infty$ iff $\rho(r) \in \{0, 1\}$.
\end{prop}

\begin{proof}
Suppose $\rho(r) \ne \infty$, and let $S = \{s_1, \ldots, s_n\} $ be a finite semi-inverse 
set for $r$. Now $\prod_{i=1}^n (1 - s_i r) = 1 - sr$ for some $s \in R$ (since the product
is finite). Thus $r(1 - sr) = r \prod_{i=1}^n (1 - s_i r) \in \rad(R)$, so $\{s\}$ is a 
semi-inverse set for $r$, and $\rho(r) \le 1$. 
\end{proof}

\begin{prop} \label{rhoZero}
Let $R$ be a ring, $r \in R$. Then $\rho(r) = 0$ iff $r \in \rad(R)$.
\end{prop}

\begin{proof}
If $r \in \rad(R)$, then $\emptyset$ is a semi-inverse set for $r$. Conversely, if 
$r \not \in \m$ for some $\m \in \mSpec(R)$, then if $S$ is any semi-inverse set 
for $r$, there must exist $s \in S$ with $1 - sr \in \m$, so $|S| \ge 1$, hence 
$\rho(r) \ge 1$.
\end{proof}

\begin{prop}
Let $R$ be a ring. Then $R^\times \subseteq \rho^{-1}(\{1\})$, and 
equality holds iff $\Spec(R/\rad(R))$ is connected.
\end{prop}

\begin{proof}
If $u \in R^\times$, then $\{u^{-1}\}$ is a semi-inverse set for $u$, so $\rho(u) = 1$
(as $u \not \in \rad(R) \implies \rho(u) \ne 0$). For the second statement, suppose
$R/\rad(R)$ has no idempotents, and pick $r \in R$, $\rho(r) = 1$. Let $\{s\}$ be 
a semi-inverse set for $r$, so $r(1 - sr) \in \rad(R)$. Then $\overline{r} = \overline{s}
\cdot \overline{r}^2$ in $R/\rad(R)$, so $\overline{s} \cdot \overline{r}$ is idempotent 
in $R/\rad(R)$. By assumption $\overline{s} \cdot \overline{r} = \overline{0}$ or 
$\overline{1}$. If $\overline{s} \cdot \overline{r} = \overline{0}$, then $\overline{r} = 
(\overline{s} \cdot \overline{r}) \overline{r} = \overline{0}$, i.e. $r \in \rad(R)$, but this 
cannot happen if $\rho(r) = 1$. Thus $\overline{s} \cdot \overline{r} = \overline{1}$, 
so $r$ is a unit modulo $\rad(R)$, hence $r$ is in fact a unit in $R$. 

Conversely, suppose $\rho^{-1}(\{1\}) = R^\times$, and let $r \in R$ 
with $\overline{0} \ne \overline{r}$ idempotent in $R/\rad(R)$. Then $r - r^2 
\in \rad(R)$, so $\{1\}$ is a semi-inverse set for $r$, i.e. $\rho(r) = 1$, so 
$r \in R^\times$. This implies $R/\rad(R)$ has only trivial idempotents, hence 
has connected spectrum.
\end{proof}

\begin{rem}
i) If $\Spec(R/\rad(R))$ is connected, then $\Spec(R)$ is also connected: if
$e \in R$ is idempotent, then $\overline{e} \in R/\rad(R)$ is also idempotent, so 
(replacing $e$ by $1 - e$ if necessary) $\overline{0} = \overline{e} \implies 
e \in \rad(R) \implies 1 - e \in R^\times$, hence $e(1 - e) = 0 \implies e = 0$.

ii) If $R$ is the coordinate ring of an (irreducible) affine variety (i.e. a finitely
generated domain over a field), then $\Spec(R/\rad(R))$ is connected. 
\end{rem}

\Cref{rhoValues} and \Cref{rhoZero} indicate that the only interesting behavior 
occurs for elements $r \in R$ with $\rho(r) = 1$, which motivates the following 
definition:

\begin{mydef}
Let $R$ be a ring, $r \in R$. If $\rho(r) = 1$, we say that $r$ is a \textit{semi-unit}.
In this case, if $\{s\}$ is a semi-inverse set for $r$, 
we say that $s$ is \textit{a semi-inverse of r}. If every element of $R$ is either a 
semi-unit or in the Jacobson radical (i.e. $\rho(R) \subseteq \{0, 1\}$), we say 
that $R$ is a \textit{semi-field}.
\end{mydef}

\begin{rem}
According to the definition, only semi-units can have semi-inverses, so although
$\{1\}$ (or indeed any singleton set) is a semi-inverse set for $0$, $1$ is not
treated as a semi-inverse of $0$. Also, the relation of being a semi-inverse need
not be symmetric: e.g. in $\Z/10\Z$, $3$ is a semi-inverse of $2$ (as $2 \equiv
3 \cdot 2^2 \bmod 10$), but $2$ is not a semi-inverse of $3$ ($3 \not \equiv
2 \cdot 3^2 \bmod 10$). However, notice that $2$ and $8$ are
semi-inverses of each other.
\end{rem}

The following proposition addresses uniqueness of semi-inverses:

\begin{prop}
Let $R$ be a ring, $r \in R$ a semi-unit. If $s_1, s_2 \in R$ are 
semi-inverses of $r$, then $s_1 - s_2 \in \rad(R) :_R r$. Conversely, if $s$ is a
semi-inverse of $r$ and $a \in \rad(R) :_R r$, then $s + a$ is a semi-inverse
of $r$.
\end{prop}

\begin{proof}
If $s_1, s_2$ are semi-inverses of $r$, then $r(1 - s_1 r), r(1 - s_2 r) \in \rad(R)$,
so $r(1 - s_1 r) - r(1 - s_2 r) = (s_2 - s_1)r^2 \in \rad(R)$, i.e. $s_2 - s_1 \in 
\rad(R) : r^2$. For the second statement, if $s$ is a semi-inverse of $r$ and
$a \in \rad(R) : r^2$, then $r(1 - sr), ar^2 \in \rad(R)$, so $r(1 - (s + a)r) = r(1 - sr)
- ar^2 \in \rad(R)$ also.

Finally, notice that $\rad(R) : r^2 = \rad(R) : r$, since if $ar^2 \in \rad(R)$, then 
$(ar)^2 = a(ar^2) \in \rad(R) \implies ar \in \rad(R)$, as $\rad(R)$ is a radical
ideal.
\end{proof}

Thus semi-inverses of $r$ are unique precisely up to cosets of 
$\rad(R) : r$. In particular, semi-inverses of non-trivial semi-units
are never unique:

\begin{cor}
Let $R$ be a ring, $r \in R$ a semi-unit. Then $r$ has a unique semi-inverse 
iff $r$ is a unit and $\rad(R) = 0$.
\end{cor}

\begin{proof}
$\Leftarrow$: if $r$ is a unit, then $\rad(R) : r = \rad(R) = 0$, so $r^{-1}$ is the 
only semi-inverse of $r$. $\Rightarrow$: if $r$ has a unique semi-inverse $s$, then 
$\rad(R) = 0$, and $r = sr^2$. But $0 = \rad(R) : r = 0 : r$, so $r$ is a 
nonzerodivisor, hence $1 = sr$, i.e. $r \in R^\times$.
\end{proof}

On the other hand, any semi-unit has a semi-inverse that is a unit. This follows 
from the following general decomposition theorem:

\begin{theorem} \label{decompThm}
Let $R$ be a ring, $r \in R$. Then $r$ is a semi-unit iff $r = ue + t$ for some 
$t \in \rad(R)$, $u \in R^\times$, and $e \in R$ a semi-unit with $1$ a 
semi-inverse of $e$ ($\iff \overline{e}$ idempotent in $R/\rad(R)$). 
In particular, $u^{-1}$ is a semi-inverse of $r$.
\end{theorem}

\begin{proof}
Passing to $R/\rad(R)$, it suffices to show that $\overline{r}$ is a product of a 
unit and an idempotent. Let $s$ be a semi-inverse of $r$, so $\overline{r} = 
\overline{s} \overline{r}^2$. Set $\overline{e} := \overline{r} \overline{s}$. Then
$\overline{e}^2 = \overline{e}$, so if $e$ is any lift of $\overline{e}$, then $e$
is a semi-unit in $R$ with $1$ as a semi-inverse. Notice also that $\overline{r}
= \overline{r} \overline{e}$. 

Next, set $\overline{u} := \overline{r} \overline{e} + (1 - \overline{e})$. Then
$\overline{u} \overline{e} = \overline{r} \overline{e}^2 + (1 - \overline{e}) 
\overline{e} = \overline{r}$. Furthermore, 
\begin{align*}
\overline{u} \cdot (\overline{s} \overline{e} + (1 - \overline{e})) &= 
(\overline{r} \overline{e} + (1 - \overline{e})) \cdot
(\overline{s} \overline{e} + (1 - \overline{e})) \\
&= \overline{r} \overline{s} \overline{e}^2 + (1 - \overline{e})^2 \\
&= \overline{e}^3 + (1 - \overline{e}) \\
&= 1
\end{align*}
so $\overline{u}$ is a unit. Lifting to $R$ gives a unit $u \in R$, such that 
$t := r - ue \in \rad(R)$. 

Finally, notice that $r(1 - u^{-1}r) = (ue + t)(1 - u^{-1}(ue + t)) = ue(1 - e) + 
t(1 - 2e - u^{-1}t) \in \rad(R)$, so $u^{-1}$ is a semi-inverse of $r$.
\end{proof}

\section{Semi-fields}

Having described the structure of semi-units, we now focus on 
the rings that have as many semi-units as possible, starting with the 
following criterion:

\begin{prop} \label{semiFieldCharAlg}
Let $R$ be a ring. Then the following are equivalent:

i) $R$ is a semi-field

ii) $R/\rad(R)$ is von Neumann regular

iii) $\dim R/\rad(R) = 0$.
\end{prop}

\begin{proof}
$R$ is a semi-field $\iff$ for every $r \in R$, there exists $s \in R$ with $r(1 - sr)
\in \rad(R) \iff$ for every $\overline{r} \in R/\rad(R)$, there exists $\overline{s}
\in R/\rad(R)$ with $\overline{r} = \overline{s} \cdot \overline{r}^2 \iff R/\rad(R)$
is von Neumann regular. Since $R/\rad(R)$ is always reduced, this happens
iff $\dim R/\rad(R) = 0$.
\end{proof}

A geometric reformulation of the semi-field property is given by:

\begin{prop} \label{semiFieldCharGeom}
Let $R$ be a ring. Then $R$ is a semi-field iff $\mSpec R$ is closed in $\Spec R$.
\end{prop}

\begin{proof}
First, note that the closure of $\mSpec R$ is equal to $V(\rad R)$: for any 
$\p \in \Spec R$, $\p$ is in $\overline{\mSpec R} \iff$ for all 
$f \in R$ with $p \in D(f)$, there exists $\m \in \mSpec R$ with $\m \in D(f) 
\iff R - \p \subseteq \bigcup_{\m \in \mSpec R} (R - \m) \iff \p \supseteq 
\bigcap_{\m \in \mSpec R} \m = \rad R$.

Thus, $\mSpec R = \overline{\mSpec R}$ iff $\mSpec R = V(\rad R)$ iff $\dim
R/\rad(R) = 0$, so the conclusion follows from \Cref{semiFieldCharAlg}.
\end{proof}

\begin{cor} \label{noethSemifield}
The following are equivalent for a ring $R$:

i) $R$ is semilocal

ii) $R/\rad(R)$ is Artinian

iii) $R$ is a semi-field and $|\Min(\rad(R))| < \infty$

\noindent
(here $\Min(\cdot)$ denotes the set of minimal primes).
\end{cor}

\begin{proof}
iii) $\implies$ ii): If $R$ is a semi-field with $\Min(R/\rad(R)) = 
\Spec(R/\rad(R))$ finite, then $R/\rad(R)$ is a von Neumann regular
ring with finite spectrum, hence is Noetherian.

ii) $\implies$ i): An Artinian ring is semilocal, and $R/\rad(R)$ semilocal
$\implies R$ semilocal.

i) $\implies$ iii): If $R$ is semilocal, then by Chinese Remainder
$R/\rad(R)$ is a finite direct product of fields.
\end{proof}

We give two ways to produce new semi-fields:

\begin{prop} \label{prodQuot}
The class of semi-fields is closed under quotients and products.
\end{prop}

\begin{proof}
Let $R$ be a semi-field, and $I$ an $R$-ideal. The surjection $p : R \to 
R/I$ sends $p(\rad(R)) \subseteq \rad(R/I)$, so $(R/I)/\rad(R/I)$ is a quotient 
of $R/(\rad(R) + I)$, which is itself a quotient of $R/\rad(R)$. Thus 
$\dim R/\rad(R) = 0$ implies $\dim (R/I)/\rad(R/I) = 0$.

If now $R_i$ are semi-fields, then by \Cref{radOfProd}
\[(\prod_i R_i)/\rad(\prod_i R_i) = (\prod_i R_i)/(\prod_i \rad(R_i)) = 
\prod_i R_i/\rad(R_i)\] is a product of von Neumann regular rings, hence is 
von Neumann regular.
\end{proof}

\begin{rem}
Geometrically, the first part of \Cref{prodQuot} says that the semi-field 
property passes to closed subschemes. However, the semi-field 
property does not pass to open subschemes -- e.g. if $R$ is any
Noetherian ring, $x \in \rad(R)$ but $x$ is not contained in any minimal 
prime of $R$, then $\dim R_x = \dim R - 1$, and if $\dim R < \infty$, then
$\rad(R_x) = \nil(R_x)$. Thus any Noetherian local domain $(R, \m)$ of 
dimension $\ge 2$ and $0 \ne x \in \m$ gives an example where $R$ is a 
semi-field (being local), but $R_x$ is not.
\end{rem}

Even in light of \Cref{prodQuot}, it is still reasonable to ask for nontrivial
examples of semi-fields. One trivial reason for being a semi-field is that the 
set of closed points is finite, and \Cref{noethSemifield} guarantees that this is the 
only possibility in the Noetherian case -- thus, one must search among 
non-Noetherian rings for a nontrivial example. 

Now one can easily form non-Noetherian rings by taking infinite products.
However, products are an arguably trivial way to construct examples -- 
for finite products, the geometric intuition is that the property of the closed 
points forming a closed set should pass to disjoint unions. This intuition fails 
for general von Neumann regular rings though, since not every von Neumann 
regular ring is a product of fields: e.g. if $k$ is a finite field, then the subring 
of $\prod_{i \in \N} k$ consisting of eventually constant sequences is 
non-Noetherian and countable, whereas any product of fields is either 
Noetherian or uncountable. Despite this, von Neumann regular rings are
trivially semi-fields for the same reason any zero-dimensional ring is: the 
set of closed points is certainly closed if every point is closed!

Nevertheless, there are indeed less trivial examples of semi-fields, which arise
formally in a manner similar to Hilbert's basis theorem and (a general form of)
the Nullstellensatz, which say that the Noetherian and Jacobson properties pass 
to rings of finite type. To emphasize the analogy, for a ring $R$, we say that 
a ring is of semi-finite type over $R$ if it is of the form $R[[x_1, \ldots, x_n]]/I$.

\begin{prop}
Let $R$ be a semi-field. Then any ring of semi-finite type over $R$ is a semi-field.
\end{prop}

\begin{proof}
By \Cref{prodQuot}, it suffices to show $R$ semi-field
$\implies R[[x_1, \ldots, x_n]]$ semi-field, and by induction it is enough to
do the base case $n = 1$. This follows immediately from the fact that 
$x \in \rad(R[[x]])$, which in turn implies that every maximal ideal of $R[[x]]$ is
of the form $\m R[[x]] + (x)$ for a (uniquely determined) maximal ideal $\m$ of 
$R$, so $R/\rad(R) \cong R[[x]]/\rad(R[[x]])$.
\end{proof}

\section{Property (*) revisited}

We finally return to the original surjectivity question. \Cref{semiFieldCharAlg} shows
that every example given earlier of a ring with $(*)$ has been a semi-field. 
The following theorem gives the general phenomenon:

\begin{theorem} \label{semifieldMainThm}
Let $R$ be a semi-field. Then $R$ has $(*)$.
\end{theorem}

\begin{proof}
By \Cref{reduceToRad} we may pass to $R/\rad(R)$, so by \Cref{semiFieldCharAlg} it suffices 
to show any von Neumann regular ring $R$ has $(*)$. For this we use 
\Cref{firstChar}(iv). Let $I \ne R$ be an ideal, and $a \in R$ such that $1 - ab \in I$
for some $b \in R$. As $R$ is von Neumann regular, $I$ is a radical ideal,
so $I = \bigcap_i p_i$ for some primes $p_i \in \Spec R$. Then $1 - ab \in p_i$
implies $a \not \in p_i$, for all $i$.
Now $a$ is a semi-unit, so by \Cref{decompThm}, $a$ has a semi-inverse which is a unit,
i.e. there exists $u \in R^\times$ with $a = a^2u$. Then $a(1 - au) = 0 \in p_i$ for all $i$, 
so $1 - au \in p_i$ for all $i$, hence $1 - au \in I$.
\end{proof}

\begin{rem}
\Cref{noethSemifield}, \Cref{prodQuot}, and \Cref{semifieldMainThm} give an alternate proof 
of \Cref{prodCase}, that an arbitrary product of semilocal rings has $(*)$. We do not know if 
the class of rings with $(*)$ is closed under arbitrary products.
\end{rem}

\Cref{semifieldMainThm} thus generalizes and gives a uniform proof of all the previous 
sufficient conditions for a ring to have $(*)$: \Cref{radCase}, \Cref{semiCase}, and 
\Cref{prodCase}. 





We conclude with an application and a generalization.
Although the motivation in determining when an ideal or ring has $(*)$ has
been mostly intrinsic, one possible application of 
these results is in constructing rings with trivial unit group. 

\begin{prop} \label{trivialUnits}
Let $X \subseteq \mathbb{P}^n_{\mathbb{F}_2}$ be a reduced 
projective scheme. Then the homogeneous coordinate ring of $X$ 
has trivial unit group.
\end{prop}

\begin{proof}
Let $S = \mathbb{F}_2[x_0, \ldots, x_n] = \Gamma_*(\mathbb{P}^n_{\mathbb{F}_2}, 
\mathcal{O}_{\mathbb{P}^n_{\mathbb{F}_2}})$ and 
$R = \mathbb{F}_2[x_0, \ldots, x_n]/I$, where $I$ is a homogeneous radical 
ideal. Then $I = p_1 \cap \ldots \cap p_m$, where $p_i$ are homogeneous 
primes in $S$, so $R \hookrightarrow S/p_1 \times \ldots \times S/p_m$. 
Thus $R^\times \subseteq \prod_{i=1}^m (S/p_i)^\times$, so it suffices to 
show $(S/p_i)^\times = \{1\}$ for each $i$. Now $S$ is a polynomial 
ring over $\mathbb{F}_2$, so $S^\times = (\mathbb{F}_2)^\times = \{1\}$,
and each $p_i \subseteq S_+$, so by \Cref{gradedCase} there is a surjection 
$\{1\} = S^\times \twoheadrightarrow (S/p_i)^\times$.
\end{proof}

In fact, the above reasoning holds in any number of variables. Thus, if 
$R = \Z[x_1, \ldots]/I$ is any ring presented as a $\Z$-algebra, then 
homogenizing the defining ideal $I$ with a new variable $x_0$ gives a 
standard graded ring $\widetilde{R} := \Z[x_0, x_1, \ldots]/\widetilde{I}$, and then 
$(\widetilde{R} \otimes_\Z \mathbb{F}_2)_{\text{red}} = 
\mathbb{F}_2[x_0, x_1, \ldots]/\sqrt{\widetilde{I}}$ has trivial unit group. 

Conversely, every ring with trivial unit group has characteristic 2 (as $1 = -1$)
and has trivial Jacobson radical (in particular, is reduced). Thus if $R^\times = \{1\}$,
then $R$ is the (affine) coordinate ring of a reduced scheme over $\mathbb{F}_2$,
and \Cref{trivialUnits} realizes every (standard) graded ring with trivial unit group.


Finally, one possible generalization is to consider other functors from
\textbf{CRing} to \textbf{Grp}. A natural choice which directly generalizes 
the group of units functor is $GL_n(\_) : \textbf{CRing} \to \textbf{Grp}$, 
which for $n = 1$ coincides with $(\_)^\times$. In order to treat the case
of $GL_n$, it is necessary to consider noncommutative rings, and nonabelian 
groups.

It is also possible to define property $(*)$ for 
two-sided ideals in a noncommutative ring. It turns out that the key place 
where commutativity was used in Section 1 was to describe the preimage of 
the units of $R/I$ as a saturation $(1 + I)^\sim$. To be precise, let us make
the following definition:

\begin{mydef}
Let $R$ be an arbitrary (possibly noncommutative) ring, and $W \subseteq R$.
Define the saturation of $W$ as
\[W^\sim := \{x \in R \mid \exists y \in R: xy, yx \in W\}\]
\end{mydef}

When $R$ is commutative, this reduces to the previous definition of $\sim$,
and one can check that $\sim$ is still a closure operation: e.g. $W \subseteq W^\sim$
follows from existence of a $1$, and idempotence (i.e. $W^\sim = (W^\sim)^\sim$)
follows from associativity of multiplication (note: if the condition ``either $xy \in W$ or 
$yx \in W$" was used instead, then $\sim$ would no longer be idempotent).

The definition of $\sim$ above agrees well with units, since units are by definition 
two-sided. For example, $\{1\}^\sim = R^\times$, and the exact statement of 
\Cref{firstChar} goes through without change. 

However,
this turns out to be unnecessary for $GL_n$, because of the fact that in the 
matrix ring, $AB = 1$ iff $BA = 1$. In other words, the definition for $\sim$ above 
works well in a Dedekind-finite ring (i.e. $xy = 1 \iff yx = 1$ for all $x, y \in R$).

\begin{prop} \label{GLn}
Let $R$ be a ring, $I \subseteq \rad(R)$ an $R$-ideal, and 
$p : R \twoheadrightarrow R/I$ the canonical surjection. Then for any 
$n \in \N$, $\overline{p} : GL_n(R) \to GL_n(R/I)$ is surjective.
\end{prop}

\begin{proof}
Pick $B = (b_{ij}) \in GL_n(R/I)$, and let $A = (a_{ij}) \in M_n(R)$ be any 
(entrywise) lift of $B$ to $R$, i.e. $p(a_{ij}) = b_{ij}$ for all $i, j$. 
Since $\det A$ is a polynomial in the entries of $A$, $p(\det A) = 
\det B$ is a unit in $R/I$. But $I \subseteq \rad(R)$, so $\det A$ is in 
fact a unit in $R$, i.e. $A \in GL_n(R)$.
\end{proof}

Notice that the proof of \Cref{GLn} shows a stronger fact than 
preserving surjectivity; namely, any lift of a matrix in $GL_n(R/I)$ is 
already in $GL_n(R)$. In fact, the analogues of \Cref{radCase} and 
\Cref{reduceToRad} hold for $GL_n(\_)$ as well, and show that \Cref{GLn}
holds for semilocal rings as well.




\vskip 4ex

\begin{thebibliography}{10}

\bibitem{E}
 Eisenbud, D. 
 \newblock {\em Commutative algebra with a view toward algebraic geometry}. 
 \newblock Grad. Texts in Math., 150. Springer-Verlag (1995).

\bibitem{GH}
Gilmer, R., and Heinzer, W. 
\newblock {\em Products of commutative rings and zero-dimensionality}. 
\newblock Trans. Amer. Math. Soc. 331.2 (1992): 663-680.

\bibitem{L}
Lam, T. Y.
\newblock {\em Lectures on modules and rings}.
\newblock Grad. Texts in Math., 189. Springer-Verlag (1999).

\end{thebibliography}
\end{document}